\documentclass[12pt]{amsart}
\usepackage{amssymb}
\usepackage{mathptmx}

\usepackage[active]{srcltx}  
\usepackage[latin1]{inputenc}
\usepackage[all]{xy}

\newcommand{\mm}{\mathfrak m}
\newcommand{\C}{\mathbb{C}}
\newcommand{\N}{\mathbb{N}}

\newcommand{\Q}{\mathbb{Q}}
\newcommand{\R}{\mathbb{R}}

\newcommand{\Z}{\mathbb{Z}}
\newcommand{\Ac}{\mathcal{A}}

 \DeclareMathOperator{\cx}{cx}

 \DeclareMathOperator{\gin}{gin}

 \DeclareMathOperator{\ini}{in}

\DeclareMathOperator{\Ker}{Ker} 
 \DeclareMathOperator{\ch}{char}

 \DeclareMathOperator{\rank}{rank}
\DeclareMathOperator{\reg}{reg}

 \DeclareMathOperator{\supp}{supp}
 
\DeclareMathOperator{\spa}{span} \DeclareMathOperator{\Tor}{Tor}

\DeclareMathOperator{\pa}{\partial}

\DeclareMathOperator{\Dirsum}{\bigoplus}

\DeclareMathOperator{\pnt}{\raise 0.5mm \hbox{\large\bf.}}

\newcommand{\fall}{\mbox{for all} ~}

\newtheorem{thm}{\bf Theorem}[section]
\newtheorem{lem}[thm]{\bf Lemma}
\newtheorem{cor}[thm]{\bf Corollary}
\newtheorem{prop}[thm]{\bf Proposition}
\newtheorem{conj}[thm]{\bf Conjecture}
\newtheorem{que}[thm]{\bf Question}

\theoremstyle{definition}

\newtheorem{exa}[thm]{\bf Example}

\title{Note on resonance varieties}

\author{Phong Dinh Thieu}
\address{Universit\"at Osnabr\"uck, Institut f\"ur Mathematik, 49069 Osnabr\"uck, Germany}
\email{pthieudi@uos.de}

\begin{document}

\begin{abstract}
We study the irreducibility of resonance varieties of graded rings over an exterior algebra $E$ with particular attention to Orlik-Solomon algebras. We  prove that for a stable monomial ideal in $E$ the first resonance variety is irreducible. If $J$ is an Orlik-Solomon ideal of an essential central hyperplane arrangement, then we show that its first resonance variety is irreducible if and only if the subideal of $J$ generated by all degree 2 elements has a $2$-linear resolution. As an application we characterize those hyperplane arrangements of rank $\leq 3$ where $J$ is componentwise linear. Higher resonance varieties are also considered.  We prove results supporting a conjecture of Schenck-Suciu relating the Betti numbers of the linear strand of $J$ and its first resonance variety. A counter example is constructed that this conjecture is not true for arbitrary graded ideals.
\end{abstract}
\maketitle
%
%
%
\section{Introduction}
\label{intro}

Let $\Ac=\{H_1,\dots,H_n\}$ be an essential central
affine hyperplane arrangement in $\C^l$ with the complement $X(\Ac)=\C^l \setminus
\bigcup_{H\in \Ac} H$. Let $E=K\langle e_1,\dots,e_n \rangle$ be the exterior algebra over a field $K$ with $\ch K=0$. In the last decades, many properties of hyperplane arrangement have been studied using the so-called the \emph{Orlik-Solomon algebra} of $\Ac$, that is the quotient ring
$E/J$ where $J$ is the \emph{Orlik-Solomon ideal} of $\Ac$ generated
by all elements
\begin{equation}
\label{generators}
\partial{e_F}
= \sum_{j=1}^t (-1)^{j-1}  e_{i_1}\wedge \dots\wedge
\widehat{e_{i_j}} \wedge\dots\wedge e_{i_t} \text{ for }
F=\{i_1,\dots,i_t\}\subseteq [n]=\{1,\dots,n\}.
\end{equation}
where $\{H_{i_1},\dots,H_{i_t}\}$ is a dependent set of hyperplanes, i.e. choosing linear forms $\alpha_i \in (\C^l)^*$ such that $\Ker \alpha_i=H_i$, then $\alpha_{i_1},\dots,\alpha_{i_t}$ are linearly dependent. Here $e_F$ is the monomial $e_{i_1}\wedge \dots \wedge e_{i_t}$ in $E$. Orlik and Solomon \cite{OS} showed that the cohomology ring of $X(\Ac)$ is entirely determined by $L(\Ac)= \{\bigcap_{H\in \Ac'} H | \Ac'\subseteq \Ac\}$, the
intersection lattice of $\Ac$. More precisely, the singular
cohomology $H^{\pnt}(X(\Ac);K)$ of $X(\Ac)$ with coefficients in $K$ is
isomorphic its Orlik-Solomon algebra. See Orlik-Terao
\cite{ORTE} and Yuzvinsky \cite{Yu} for details. See also, e.g.,
\cite{AAH, DeYu, EPY, GT, NRV, SS1, SS2} for the study of Orlik-Solomon algebras
via exterior algebra methods and algebraic properties of arbitrary modules over $E$.

Falk \cite{F} defined \emph{resonance varieties} to study the ring structure of Orlik-Solomon algebras which have shown to be useful in the recent years. For a graded
algebra $A=E/J$ where $J$ is a graded ideal of $E$ and $u\in A_1$, we have a cochain complex
$$
(A,u): \qquad 0\longrightarrow A_{0}\stackrel{\cdot u }{\longrightarrow}
A_{1}\stackrel{\cdot u}{\longrightarrow}\ldots \stackrel{\cdot u }{\longrightarrow}A_{r}\stackrel{\cdot
u}\longrightarrow \ldots
$$
since $u^2=0$. Its cohomology is denoted by
$H^{\pnt}(A,u)$. The $p$-th resonance variety of $A$ is
$$
R^p(A)=\{u\in A_1: H^p(A,u)\neq 0\} \text{ for } p\geq 0.
$$
It is known that $R^p(A)$ is an affine variety in
$A_1\cong K^{\dim_K A_1}$. See \cite{LY, SCDFSTY} for more details.

Let $M$ be a finitely generated graded left and right $E$-module satisfying the equations $u m=(-1)^{\deg u\deg m}mu$ for homogeneous elements $u\in E$, $m\in M$. The graded Betti numbers of $M$ are $\beta^E_{i,j}(M)=\dim_K\Tor^E_i(K,M)_{j}$. We say that $M$ has a \emph{$d$-linear resolution} if $\beta^E_{i,i+j}(M)=0$ for all $i$ and $j\neq d$. Following \cite{HH} we call $M$ \emph{componentwise linear} if all submodules $M_{\langle i\rangle}$ of $M$ generated by $M_i$ have an $i$-linear resolution for $i\in \Z$. Let $d$ be the initial degree of $M$, i.e. $M_i = 0$ for $i < d$ and $M_d \neq 0$.  We have $\beta^E_{i,i+j}(M)=0$ for $j<d$. The numbers $\beta^E_{i,i+d}(M)$ describe the \emph{linear strand} of the minimal graded free resolution of $M$,
i.e., they count the number of linear syzygies appearing in the
resolution.

For an Orlik-Solomon algebra $A=E/J$  of an essential central hyperplane arrangements, some results on resonance varieties are known with special attention to $R^1(A)$. For example Libgober and Yuzvinsky \cite{LY} proved that two irreducible components of $R^1(A)$ meet only at 0. Falk \cite{F} showed  that if $u,v$ belong to the same irreducible component of $R^1(A)$ then $u\wedge v\in J$.  If these two properties hold for the first resonance variety of a graded algebra $A=E/J$, we say that $A$ satisfies \emph{property (*)}.

At first,  we investigate in Section \ref{sect irrestab} the irreducibility of resonance varieties of graded algebras $E/J$.  In the case $J$ is a stable monomial ideal, we show that the first resonance variety of $E/J$ is irreducible.  This can be seen as the generic case since the generic initial ideal of any graded ideal is stable. We also prove some properties of higher degree resonance varieties and suggest a question for the irreducibility of higher degree resonance varieties. If  $J$ is an Orlik-Solomon ideal, we prove in Section \ref{sect irreOS} that the degree 2 component ideal $J_{\langle 2\rangle}$ of $J$ has a 2-linear resolution if and only if the first resonance
variety of $E/J$ is irreducible. In particular, if $J$ is componentwise linear, then the first resonance is
irreducible. The converse holds for arrangements whose rank
less than or equal to 3.

Schenck and Suciu suggested in \cite{SS2} a conjecture about the Betti numbers of the linear strand of an Orlik-Solomon algebra $A=E/J$. More precisely, the Betti numbers $\beta^E_{i,i+1}(E/J)$ should be determined by using invariants from $R^1(A)$: Observe that Orlik-Solomon ideals are generated by products of linear forms since
\begin{equation}
\label{generators2}
\partial{e_F}
= (e_{i_2}-e_{i_1})\wedge \ldots \wedge(e_{i_t}-e_{i_1}) \text{ for }
F=\{i_1,\dots,i_t\}\subseteq [n].
\end{equation}
Following \cite{DeYu} we call ideals generated by products of linear forms \textit{pure ideals}.  Note that monomial ideals are pure ideals, but not all monomial ideals satisfy property (*); see Example \ref{exconB}. A direct generalization of Conjecture B in \cite{SS2} is:
\begin{conj}
\label{conB}
Let $J\subset E$ be a pure ideal such that $E/J$ satisfies property (*). Then for $i\gg 0$, the graded Betti numbers of the linear strand of $E/J$ are given by
$$
\beta^E_{i,i+1}(E/J)= i\sum_{r\geq 1} h_r \binom{r+i-1}{i+1},
$$
where $h_r$ is the number of $r$-dimensional components of $R^1(E/J)$ in the affine space $K^n$.
\end{conj}
In Example \ref{exconB}, we show that property (*) can not be omitted. More precisely, Conjecture \ref{conB} does not hold for arbitrary pure ideals.  We prove in Theorem \ref{thmconB} that there is a class of algebras induced by certain graphs in which the conjecture hold.

We are grateful to Tim R\"omer for generously suggesting problems, many insightful
ideas on the subject of this note.
\section{Preliminaries}
\label{sect Prelimi}

In this section we recall some definitions and facts about the exterior algebra and hyperplane arrangement. Let $M$ always be a finitely generated graded left and right $E$-module satisfying the equations $u m=(-1)^{\deg u\deg m}mu$ for homogeneous elements $u\in E$, $m\in M$. Its minimal graded free resolution is an exact sequence
$$
\ldots\longrightarrow \Dirsum_{j\in\Z}E(-j)^{\beta^E_{1,j}(M)} \longrightarrow
\Dirsum_{j\in\Z}E(-j)^{\beta^E_{0,j}(M)}\longrightarrow M\longrightarrow 0.
$$
We see that the resolution is $d$-linear (as defined in Section \ref{intro}) for
some $d \in \Z$ if and only if it is of the form
$$
\ldots\longrightarrow E(-d-2)^{\beta^E_{2,d+2}(M)}\longrightarrow
E(-d-1)^{\beta^E_{1,d+1}(M)} \longrightarrow
E(-d)^{\beta^E_{0,d}(M)}\longrightarrow M \longrightarrow 0.
$$

The \emph{regularity} of $M$ is defined as
$\reg M= \max\{j-i : \beta^E_{i,j}(M)\neq 0\}$. The \emph{complexity} of $M$,
which measures the growth rate of the Betti numbers of $M$, is
defined as
$$
\cx  M=\inf\{c \in \N: \beta^E_i(M)\leq \alpha i^{c-1} \fall i\geq 1,
\alpha \in\R\}.
$$
Recall that a componentwise linear module which is generated in one degree has a linear resolution. A module that has a linear resolution is componentwise linear.

Next we present some facts about stable (strongly stable) monomial ideals and generic initial ideals.
Let $u=e_F\in E$ be a monomial where $F\subseteq [n]$. We denote $\max(u)=\max\{i: i\in F\}$. A monomial ideal $J\subseteq E$ is called \emph{stable} if $e_j\frac{u}{e_{\max(u)}}\in J$ for every monomial
$u\in J$ and $j<\max(u)$. The ideal $J$ is called \emph{strongly
stable} if $e_j\frac{u}{e_i}\in J$ for every monomial $u=e_F\in J$,
$i\in F$ and $j<i$.

For a monomial ideal $J\subset E$ let $G(J)$ be the minimal set of monomial generators of $J$, and
$G(J)_j\subseteq G(J)$ be the subset of generators in $G(J)$ of degree $j$. Aramova, Herzog and Hibi \cite{AHH} computed  a formula for the graded Betti numbers of stable ideals:
\begin{lem}\cite[Corollary 3.3]{AHH}
\label{betti}
Let $0\neq J\subset E$ be a stable monomial ideal. Then
$$
\beta^E_{i,i+j}(J)=\sum_{u\in G(J)_j}\binom{\max(u)+i-1}{\max(u)-1} \
\ \fall i\geq 0, j\in\Z.
$$
\end{lem}
The complexity of a stable monomial ideal $J$ can be interpreted in terms of
$G(J)$.
\begin{prop} \cite[Proposition 3.4]{GT}
\label{cx}
Let $0\neq J\subset E$ be a stable monomial ideal. Then
$$
\cx  E/J=\max\{\max(u):u\in G(J)\}.
$$
\end{prop}
In particular, if $J$ is stable and generated in one degree, it has
a linear resolution. An example for such an ideal is the
maximal graded ideal $\mm=(e_1,\ldots,e_n)$ of $E$ and all its
powers.

Let $<$ be a reverse lexicographic order on $E$ with $e_1 > e_2 >\ldots > e_n$. The
initial ideal of a graded ideal $J\subset E$ is the ideal generated
by the initial terms $\ini(f), f \in J$ with respect to this order,
and is denoted by $\ini(J)$. In the exterior algebra over an
infinite field, Aramova, Herzog and Hibi in \cite[Theorem 1.6]{AHH}
proved the existence of a non-empty Zariski-open subset $U \subseteq
GL(n; K)$ such that there is a monomial ideal $I\subseteq E$ with $I
= \ini(g(J))$ for all $g\in U$. This ideal $I$ is called the
\textit{generic initial ideal} of $J$, denoted by $\gin(J)$. The
generic initial ideal of a graded ideal is strongly stable if it
exists \cite[Proposition 1.7]{AHH}. This is independent of the
characteristic of the field in contrast to ideals in the polynomial
ring. In addition, if the graded ideal $J$ is componentwise linear,
then $J$ and $\gin(J)$ have the same graded Betti numbers ( see\cite[Theorem 2.1]{ArHeHi}).

Let $u=\sum_{k=1}^n \alpha_k e_k$ be a linear form in $E$. We call the set $\supp(u)= \{k: \alpha_k \neq 0\}$ the \emph{support} of $u$. Let $e_F$ be a monomial in $E$  where $F=\{i_1,\ldots,i_t\}\subseteq [n]$ and $1\not\in F$. Using Equation \eqref{generators2} one can check that
\begin{equation}
\label{formeleF}
\partial  e_F=(e_{i_2}-e_{i_1})\wedge\ldots\wedge(e_{i_t}-e_{i_1})=\sum_{j=1}^t(-1)^{j-1}\pa e_{F\setminus \{i_j\}\cup\{1\}}.
\end{equation}
Next we collect some facts and results about the intersection lattice and resonance varieties used in the following. Let $\Ac=\{H_1,\dots,H_n\}$ be an essential central affine
hyperplane arrangement in $\C^l$ with the intersection lattice $L(\Ac)$. Let $J$ be the Orlik-Solomon ideal and $A=E/J$ the Orlik-Solomon algebra of $\Ac$. We denote by $J_i$ the set of all homogeneous elements of degree $i$ of $J$ and by $J_{\leq i}$ the ideal generated by all homogeneous elements of degree $\leq i$ of $J$.

Observe that $L(\Ac)$ is a partially-ordered set whose elements are the linear subspaces of $\C^l$ obtained as intersections of sets of hyperplanes from $\Ac$ and ordered by reverse inclusion. The intersection lattice $L(\Ac)$ is a ranked
poset. Indeed, $\rank (X)$ is the codimension of $X$ in $\C^l$ for $X\in L(\Ac)$ and $\rank(\Ac)$ is the maximal value of $\{\rank(X): X\in L(\Ac)\}$. See \cite[Section 1.2]{SCDFSTY} for details. Note that if $X=H_{i_1}\cap\dots\cap H_{i_t}$ and $\rank(X)< t$ then $\{H_{i_1},\dots, H_{i_t}\}$ is a dependent set. In particular, if $\rank(\Ac)= r$ then all sets of more than $r$ hyperplanes are dependent sets and then $J_{\leq r}= J$.

Since every set of two hyperplanes are independent, $J_1=0$. The resonance varieties of Orlik-Solomon algebra $A$ can be
computed by the following formulas
\begin{eqnarray*}
R^1(A)&=&\{ u\in E_1: u=0 \text{ or there exists } v\in E_1, 0\neq u\wedge v\in J_2\},\\
R^p(A)&=&\{u\in E_1:  u=0 \text{ or there exists } v\in E_p, v\not\in J_p\cup uE_{p-1}, 0\neq
u\wedge v\in J_{p+1}\}.
\end{eqnarray*}

As shown in \cite{LY}, $R^1(A)$ is a variety in the affine space
$E_1=K^n$. Each component of $R^1(A)$ is a linear subspace of
$K^n$ (see \cite{F}). In other words, $R^1(A)$ is the union of a
subspace arrangement in $K^n$. As shown by Libgober and Yuzvinsky in
\cite{LY}, each subspace of $R^1(A)$ has dimension at least 2, two
distinct subspaces meet only at 0, and $R^p(A)$ is the union of
those subspaces of dimension greater than $d$, which can have
none-zero intersection. In \cite{F} Falk proved that if $u,v$
belong to the same irreducible component of $R^1(A)$, then $u\wedge
v\in J_2$. Falk also showed that, for each $X\in L_2(\Ac)$ which is the intersection of more than two hyperplanes, there is a
corresponding irreducible component of $R^1(A)$, called \emph{local
component} which is defined by
$$
L_X=\{(x_i)\in E_1=K^n: x_i= 0 \text{ if }
X\nsubseteq H_i \text{ and } \sum_{H_i\supseteq X} x_i=0\}.
$$
Moreover, the results in \cite[Theorem 4.46, Corollary
4.49]{SCDFSTY} and \cite[Theorem 3.1]{AAH} imply that
$R^p(A)\subseteq R^q(A)$ for $p<q\leq \rank \Ac$. We denote
by $V_E(A)$ the set of all elements $u$ of $A_1$ such that the set of
elements of $A$ annihilated by $u$ is not the same as $uA$. Then it follows that $R^p(A)\subseteq V_E(A)$ for all $1\leq p\leq \rank\Ac$. Moreover, $V_E(A)$ is a linear subspace of $E_1$ and $\dim_K V_E(A)=\cx_E(A)$. We refer to Aramova, Avramov, and Herzog
\cite{AAH} for more details, where $V_E(A)$ is called the \emph{rank
variety} of $A$.

\section{Resonance varieties of stable monomial ideals}
\label{sect irrestab}
In this section we investigate the irreducibility of  resonance varieties $R^p(E/J)$ of $E/J$ where $J$ is a stable monomial ideal. Since the generic initial ideal of a graded ideal is always stable, the class of stable monomial ideals can be seen as the generic case. Recall that we always assume that a graded $J$ is always non trivial and contains no variable.

\begin{thm}
\label{thmstrong} Let $J\subset E$ be a stable monomial ideal. Then the first resonance variety $R^1(E/J)$ of $E/J$ is irreducible.
\end{thm}
\begin{proof} Let $t=\max\{\max(u) : u\in G(J)_2\}$.
There exists a integer $r$ with $1\leq r<t$ such that $u=e_r\wedge e_t\in J_2$. We claim that $\spa_K\{e_1, \ldots, e_t\}= R^1(E/J)$ which implies that $R^1(E/J)$ is irreducible.

At first note that
$e_i \wedge u/e_t \in J \text{ for all } i<t$, since $J$ is stable. Hence
$$
e_r\wedge e_i\in J_2
\text{ for } 1\leq i\leq t
\text{ and thus }
e_r\wedge \sum_{i=1}^t \alpha_i e_i \in J_2 \text{ for every } \alpha_i\in K.
$$
Observe that
$$
R^1(E/J)=\{ u \in E_1: u=0 \text{ or there exists } v\in E_1,\ 0\neq u\wedge v\in J_2\}.
$$
Then it follows  $\spa_K\{e_1, \ldots, e_t\}\subseteq R^1(E/J)$. In particular, for $t=n$ we see that equality holds, so assume $t<n$ in the following.

We consider an arbitrary element $0\neq u=\sum_{i=1}^n \alpha_i e_i\in R^1(E/J)$. Suppose that there exists an integer $s$ and $0\neq v=\sum_{j=1}^n \beta_j e_j\in E_1$ with
$$
t<s\leq n
 \text{ such that } \alpha_s\neq 0 \text{ and } 0\neq
u\wedge v\in J_2.
$$
By the choice of $t$ we see that
$$
e_s\wedge e_j\not\in J_2 \text{ for } j\in \{1,\ldots,n\}\setminus \{s\}.
$$
Thus the monomial $e_s\wedge e_j$ can not appear in $u\wedge v$. We get
$$
\alpha_s \beta_j-\alpha_j\beta_s = 0 \text{ for } j\in \{1,\ldots,n\}.
$$
If $\beta_s=0$ then $\beta_j=0$ for $j=1,\ldots,n$. This is a contradiction
to the fact that  $v\neq 0$. So $\beta_s\neq 0$. This implies that $\alpha_j\neq 0$ if and only if $\beta_j\neq 0$ and in this case
$$
\beta_j= \frac{\beta_s}{\alpha_s}\alpha_j
$$
Thus $v=k u$ for $k=\beta_s/\alpha_s$ and we see that $u\wedge v= 0$. This is also a contradiction to choice of $u$ and $v$. We get  $\alpha_s=0$ for every integer $s$ with $s>t$. Altogether we see that
$$
R^1(E/J) \subseteq \spa_K\{e_1,\ldots,e_t\}
\text{ and then }
R^1(E/J) =\spa_K\{e_1,\ldots,e_t\}.
$$
This concludes the proof.
\end{proof}

Theorem \ref{thmstrong} motivates the following question:
\begin{que}
Let $J\subset E$ be a stable monomial ideal. Are all resonance varieties $R^p(E/J)$ of $E/J$ irreducible for $p\geq 1$?
\end{que}

We have only little knowledge about $R^p(E/J)$ for $p>1$ which will be presented below. For the result we need at first:
\begin{lem}
\label{prohire}
Let $J\subset E$ be a stable monomial ideal and let
$$
t_p=\max\{\max(u): u\in G(J),\ \deg u = p\} \text{ for }  1\leq p \leq n.
$$
Then
$$
\spa_K\{e_1,\ldots,e_{t_p}\}\subseteq R^{p-1}(E/J) \text{ for } 1\leq p
\leq\max\{\deg u: u\in G(J)\}.
$$
\end{lem}
\begin{proof}
Let $u= e_{i_1}\wedge\ldots\wedge e_{i_{p}}\in G(J)_p$ where $1\leq i_1<\dots<i_p=t_p$. Since $G(J)$
is the minimal set of generators of $J$ and $u\in G(J)_p$, we have $u\not\in \mm J_{p-1}$. Therefore
$$
e_{i_1}\wedge\ldots\wedge
e_{i_{q-1}}\wedge\widehat{e_{i_q}}\wedge e_{i_{q+1}}\ldots \wedge e_{i_{p-1}}\wedge e_{i_p}\not\in J_{p-1} \text{ for } 1\leq q\leq p.
$$
Observe that $R^{p-1}(E/J)$ equals to the set
$$
\{u \in E_1: u=0 \text{ or there exists } v\in E_{p-1}, v\not\in J_{p-1}\cup uE_{p-2} \text{ and } 0\neq u\wedge v \in J_p\}.
$$
This implies $e_{i_q}\in R^{p-1}(E/J)$ for
$q=1,\ldots,p$.

Next we consider $i \in [t_p]\setminus \{i_1,\dots,i_p\}$. Since $J$ is stable and $t_p=\max(u)$ we have
$$
0\neq e_i \wedge (u/e_{t_p}) \in J_p \text{ and } u/e_{t_p}\not\in
J_{p-1}.
$$
Hence $e_i\in R^{p-1}(E/J)$. So $\{e_1,\dots,
e_{t_p}\}\subseteq R^{p-1}(E/J)$. Let
$$0\neq v=\sum_{j=1}^{t_p}
\alpha_j e_j \in \spa_K\{e_1,\dots,
e_{t_p}\}
$$
be an arbitrary element. Assume at first that $v \not\in \spa_K\{e_{i_1},\dots,
e_{i_p}\}$. This implies
$$
0\neq v \wedge (u/e_{t_p}) \in J_p.
$$
So $v\in R^{p-1}(E/J)$. Next we assume that $v\in
\spa_K\{e_{i_1},\ldots,e_{i_p}\}$ and $\alpha_{i_q}\neq 0$ for some $1\leq q \leq p$. Then
$$
0\neq v\wedge (u/e_{i_q})=\alpha_{i_q} u \in J_p.
$$
Again we see that $v\in R^{p-1}(E/J)$. Hence $\spa_K\{ e_1,\ldots,e_{t_p}\}\subseteq R^{p-1}(E/J)$ as desired.
\end{proof}
\begin{cor}
\label{cor:maximal}Let $J \subset E$ be a stable monomial ideal generated in one degree $p\geq 2$. Then the $(p-1)^{th}$ resonance variety of $E/J$ is maximal, i.e. $R^{p-1}(E/J)= V_E(E/J).$

In particular, $R^{p-1}(E/J)$ is irreducible.
\end{cor}
\begin{proof} Let $t=\max\{\max(u): u\in G(J)\}$. With Lemma \ref{prohire} we see
$\spa_K\{e_1,\ldots,e_t\}\subseteq R^{p-1}(E/J)$. In addition, by \cite[Theorem 3.1 (2)]{AAH} and Proposition  \ref{cx} we know
$$
\dim_K V_E(E/J)= \cx_E(E/J)=t.
$$
Since
$$
\spa_K\{e_1,\ldots,e_t\}\subseteq R^{p-1}(E/J)\subseteq V_E(E/J)
\text{ and }
\dim_K V_E(E/J)=t
$$
we get that $R^{p-1}(E/J)= V_E(E/J)$.
\end{proof}

\section{Resonance varieties of hyperplane arrangements}
\label{sect irreOS}

The purpose of this section is to present results related to the question whether resonance varieties  of hyperplane arrangements are irreducible. At first we consider the first resonance variety. For the first main result we need the following observation.

\begin{lem}
\label{lem:linearforms}
Let $J=(l_1,\dots,l_t) \subset E$ be an ideal generated by $t$ linearly independent 1-forms $l_1,\dots,l_t$. Then $J^d$ has an $d$-linear resolution for all integers $d\geq 1$.
\end{lem}
\begin{proof}
After an appropriate change  of coordinates, we may assume that $J=(e_1,\dots,e_t)$. Then for a fixed integer $d\geq 1$, we have
$$
J^d=(e_1,\dots,e_t)^d=(e_F: F\subseteq \{1,\dots,t\}, |F|=d).
$$
Observe that $J^d$ is a stable monomial ideal of $E$ which is generated in degree $d$. Therefore $J^d$ has an $d$-linear resolution (see, e.g., \cite[Corollary 3.4 (a)]{AHH}).
\end{proof}

\begin{thm}
\label{thmirre2}
Let $\Ac$ be an essential central hyperplane arrangement with Orlik-Solomon ideal $J$ and Orlik-Solomon algebra $A=E/J$. The following statements are equivalent:
\begin{enumerate}
\item
The first resonance variety $R^1(A)$ of $A$ is irreducible;
\item
The ideal $J_{\langle 2 \rangle }$ has a 2-linear resolution.
\end{enumerate}
In particular, if $J$ is componentwise linear, then $R^1(A)$ is irreducible.
\end{thm}
\begin{proof}
(i) $\Rightarrow$ (ii): Assume that $R^1(A)$ is irreducible. Since elements of $L_2(\Ac)$, which are intersections of more than two hyperplanes correspond to the local components of $R^1(A)$ as noted above (see  \cite{F}), there is exactly one element $X$ in $L_2(\Ac)$, which is an intersection of more than two hyperplanes. We choose a maximal integer $s$ with $3\leq s\leq n$
such that $X$ is the intersection of $s$ hyperplanes of the arrangement. Without loss of generality we assume that $\Ac=\{H_1,\dots,H_s,H_{s+1},\dots,H_{n}\}$ and $X= H_1\cap H_2\cap\ldots\cap
H_s$.

Let $F=\{i,j,k\} \subseteq \{1,\dots,s\}$ with $|F|=3$. Since
$$
2 = \rank (H_1\cap H_2\cap\ldots\cap
H_s) \geq  \rank  (H_i \cap H_j \cap H_k) \geq 2
$$
we get that $H_i \cap H_j \cap H_k= H_1\cap H_2\cap\ldots\cap
H_s$ and thus $F$ is a dependent set of $\Ac$. Next we assume that $G=\{i,j,k\}\subseteq \{1,\dots,n\}$ with $|G|=3$ where for example $i\geq s+1$.  If $G$ is dependent, then $H_i\cap H_j\cap H_k$ would have rank 2 which implies by our assumption on $L_2(\Ac)$ that $H_i\cap H_j\cap H_k=X$. But then it would follow that
$$
X=
H_1\cap H_2\cap\ldots\cap
H_{s}=H_1\cap H_2\cap\ldots\cap
H_{s}\cap H_i\cap H_j\cap H_k
$$
which is a contradiction to the choice of $s$. Hence
\begin{eqnarray*}
J_{\langle 2\rangle}
&=& (\partial e_F: F \text{ is dependent}, |F|=3)\\
&=& ((e_i-e_k)\wedge (e_j-e_k) : \{i,j,k\} \text{ is dependent for pairwise distinct } 1\leq i,j,k \leq s )\\
&=& ((e_i-e_1)\wedge (e_j-e_1) : \{1,i,j\} \text{ is dependent for pairwise distinct } 2\leq i,j\leq s )\\
&=& (e_2-e_1,\ldots,e_s-e_1)^2.
\end{eqnarray*}

Note that we used at the third equation Formula \eqref{formeleF}. It follows from Lemma \ref{lem:linearforms} that $J_{\langle 2\rangle}$ has a $2$-linear resolution.

(ii) $\Rightarrow$ (i):
Since $J_{\langle 2 \rangle}$ has 2-linear resolution, it has regularity 2 as well as $\gin(J_{\langle 2 \rangle})$. In particular, $\gin(J_{\langle 2 \rangle})$ is generated in degree 2. Moreover, $J_{\langle 2 \rangle}$ and $\gin( J_{\langle 2 \rangle})$ have the same graded Betti numbers; see \cite[Theorem 2.1]{ArHeHi}. Note that $\gin(
J_{\langle 2\rangle})$ is a strongly stable monomial ideal and $G(\gin( J_{\langle 2\rangle}))_2=G(\gin( J_{\langle 2\rangle}))$. By Lemma \ref{betti} we get
$$
\beta^E_{i,i+2}(J_{\langle 2\rangle})=\beta^E_{i,i+2}(\gin(J_{\langle 2\rangle}))= \sum_{u\in G(\gin( J_{\langle 2\rangle}))}\binom{\max(u)+i-1}{\max(u)-1}.
$$
We consider the polynomial function
$$
P\colon \Q \to \Q,\
P(i)=\sum_{u\in G(\gin( J_{\langle 2\rangle}))}\binom{\max(u)+i-1}{\max(u)-1}.
$$
Observe that  $\deg
P=t-1$ where $t=\max\{\max(u): u\in G(\gin( J_{\langle 2\rangle}))\}$. It is a consequence of \cite[Theorem 4.3]{SS1} that $\deg P
=\dim R^1(A)-1$.  Recall that we consider $R^1(A)$ as an affine variety in $E_1=K^n$ while in \cite{SS1} this space is viewed as a projective variety.
It follows  $\dim R^1(A)=t$. As noted above $R^1(A)$ is the union of linear components $L_j$. There exists one linear component, say $L_p$, of $R^1(A)$
such that
$\dim L_p=t$. By \cite[Theorem 5.6]{SS1} (see also Section \ref{sect Prelimi}) we have for $i\gg 0$ that
$$
\beta^E_{i,i+2}(J_{\langle 2\rangle})\geq \sum_{L_j \text{ component of }
R^1(\Ac)}(i+1) \binom{\dim L_j+i}{i+2}\geq (i+1)\binom{t+i}{i+2}.
$$
Since  $\gin (J_{\langle 2\rangle})$ and $(e_1,\ldots,e_t)^2$ are both strongly stable monomial ideals generated in degree 2 and  by the definition of $t$ we get $G(\gin (J_{\langle 2\rangle}))\subseteq
G(( e_1,\ldots,e_t)^2)$ we see with Lemma \ref{betti} that
$$
\beta^E_{i,i+2}(\gin(
J_{\langle 2\rangle}))\leq \beta^E_{i,i+2} (( e_1,\ldots,e_t)^2) \text{ for all } i\geq 0.
$$
Lemma \ref{betti} and a direct computation shows that
$$
\beta^E_{i,i+2} (( e_1,\ldots,e_t)^2)
=
(i+1)\binom{t+i}{i+2}\text{ for all } i\geq 0.
$$
(This equation is, e.g., a consequence from \cite[Proposition 6.12]{GT}.) Using all inequalities together we get that
$$
\beta^E_{i,i+2}
(J_{\langle 2\rangle})=\beta^E_{i,i+2}(\gin( J_{\langle 2\rangle}))
=(i+1)\binom{t+i}{i+2} \text{ for } i\gg 0.
$$
Using again \cite[Theorem 5.6]{SS1} this implies that
$R^1(A)$ has exactly one irreducible component. Thus
$R^1(A)$ is irreducible.
\end{proof}

If the rank of the arrangement is small, we get:

\begin{cor}
\label{oscom}
Let $\Ac$ be an essential central hyperplane arrangement with Orlik-Solo\-mon ideal $J$ and Orlik-Solo\-mon algebra $A=E/J$ such that $\rank (\Ac)\leq 3$. The following statements are equivalent:
\begin{enumerate}
\item
The first resonance variety $R^1(A)$ of $A$ is irreducible;
\item
$J$ is componentwise linear.
\end{enumerate}
\end{cor}
\begin{proof}
(i) $\Rightarrow$ (ii):
Since $R^1(A)$ is irreducible, we get that
$J_{\leq 2}=J_{\langle 2 \rangle}$ has 2-linear resolution and thus $\reg (J_{\leq 2})=2$.

We have $J=J_{\leq 3}$ because $\rank (\Ac)\leq 3$. It follows from \cite[Corrolary 6.7]{GT} that
$$
\reg (J_{\leq 3})=\reg (J)=\reg(E/J)+1 \leq 3.
$$
Moreover, $\reg (J_{\leq k})\leq 3\leq k$ for $k\geq 3$.
Now it follows from \cite[Theorem 5.3.7]{G} that $J$ is componentwise linear.

(ii) $\Rightarrow$ (i):
If $J$ is componentwise linear, then $J_{\langle 2\rangle}$ has 2-linear resolution. Hence Theorem \ref{thmirre2} implies that $R^1(A)$ is irreducible.
\end{proof}

There exists Orlik-Solomon ideals which are componentwise linear, but do not have a linear resolution as the following example shows.
\begin{exa} Let $\Ac$ be an essential central
hyperplane arrangement in $\C^3$ with defining polynomial
$$
Q=xy(x-y)z(2x+y-z)(x+3y+z).
$$
Let $E=K\langle e_1,\dots,e_6\rangle$ be the exterior algebra where each $e_i$ corresponds to $i$-th factor in the equation of the polynomial. The Orlik-Solomon ideal of $\Ac$ is
$$
J=(\partial e_{123})+(\partial e_{ijkl} :  \{i,j,k,l\} \subseteq [6]).
$$
We see that $L_2(\Ac)$ has only one element $X= H_1\cap H_2\cap H_3$ such that $|X|\geq 3$. Hence $R^1(A)= \spa_K\{(e_2-e_1), (e_3-e_1)\}$ is irreducible.
By Corollary \ref{oscom}, the ideal $J$ is a componentwise linear ideal. We observe that the elements $\partial e_{ijkl}$ are not redundant for all $1\leq i,j,k,l\leq 6$, so $J\neq J_{\langle 2\rangle}$. This implies that $J$ is not generated in one
degree. Hence $J$ has not a linear resolution.
\end{exa}

We saw that the property componentwise linear of an Orlik-Solomon ideal can be characterized in terms of data of the hyperplane arrangement if the rank is small. We wonder if a similar statement can be proved for arbitrary essential central hyperplane arrangements. Note that a characterization of having a linear resolution is given in \cite[Corollary 3.6]{EPY}; see also \cite[Theorem 6.11]{GT} which is a first step to such a result.

We ask ourself the following
\begin{que}
Assume that the Orlik-Solomon ideal $J$ of an essential central hyperplane
arrangement $\Ac$ is componentwise linear. Are then all
resonance varieties $R^p(A)$ where $0\leq p\leq \rank(\Ac)$ irreducible?
\end{que}

Another corollary from Theorem \ref{thmirre2} is:
\begin{cor}
\label{cor:2linearconj}
Let $\Ac$ be an essential central hyperplane arrangement with Orlik-Solo\-mon ideal $J$ and Orlik-Solo\-mon algebra $A=E/J$ such that $J_{\langle 2 \rangle}$ has a $2$-linear resolution. Then Conjecture \ref{conB} is true for $E/J$.
\end{cor}
\begin{proof}
Let $t=\max\{\max(u): u\in G(\gin( J_{\langle 2\rangle}))\}$. In the proof of Theorem \ref{thmirre2} we showed that
$$
\beta^E_{i,i+2}(J)=\beta^E_{i,i+2}(J_{\langle 2\rangle})=(i+1)\binom{t+i}{i+2} \text{ for } i\gg 0.
$$
We know also that $h_r=1$ for $r=\dim R^1(E/J)=t$ and $h_r=0$ for $r\neq t$ since $R^1(E/J)$ is irreducible. This concludes the proof.
\end{proof}

If the Orlik-Solomon ideal has a linear resolution, then we can prove analogously to Corollary \ref{cor:maximal}:

\begin{prop}
\label{mlr}
Let $\Ac$ be an essential central hyperplane arrangement with Orlik-Solomon ideal $J$ and Orlik-Solomon algebra $A=E/J$
such that $J$ has a $d$-linear
resolution. Then
$$
R^{p}(A) = 0 \text{ for } 0\leq p \leq d-2 \text{ and } R^{d-1}(A) = V_E(A) \text{ is irreducible}.
$$
\end{prop}
\begin{proof}
Assume that $J$ has a $d$-linear free resolution. Then
$J_0=\ldots = J_{d-1}=\{\}$ and $J=(J_d)$. Therefore
$R^p(A)=0$ for $0\leq p\leq d-2$. By \cite[Theorem 6.11 (iii)]{GT}, the
matroid $M$ of $\Ac$ is $M= U_{d, n-f}\oplus U_{f,f}$ where $U_{p,q}$ is a uniform matroid of rank $p$ whose the ground set has $q$ elements and all subsets of $[q]$ of cardinality $\leq
p$ are independent. Therefore
\begin{eqnarray*}
J=J_{\langle d\rangle}
&=&(\partial e_F: F\in U_{d,n-f}, |F|=d+1)\\
&=& ((e_{i_2}-e_{i_1})\wedge\ldots\wedge (e_{i_{d+1}}-e_{i_1}) :
F=\{i_1,\dots,i_{d+1}\}\subseteq [n-f]
)\\
&=& ((e_{i_2}-e_{1})\wedge\ldots\wedge (e_{i_{d+1}}-e_{1}) :
F=\{1,i_2,\dots,i_{d+1}\}\subseteq [n-f]
).
\end{eqnarray*}
Note that we used at the third equation Formula \eqref{formeleF}. Thus
\begin{eqnarray*}
R^{d-1}(A)&=&
\{u\in E_1:  u=0 \text{ or there exists } v\in E_{d-1}, v\not\in u E_{d-2}, 0\neq
u\wedge v\in J_{d}\}\\
&=& \spa_K\{e_i-e_1: 2\leq i\leq n-f\}.
\end{eqnarray*}
This implies already that $R^{d-1}(A)$ is irreducible and $\dim_K R^{d-1}(A) = n-f-1$.

By \cite[Corollary 6.7]{GT} we have that $\cx_E(A)$ is equal to the
$n$ minus the number of components of the matroid of $\Ac$. Together with \cite[Theorem 3.2]{AAH} we get
$$
\dim_K V_E(A)= \cx_E(A)-f-1=\dim_K R^{d-1}(A).
$$
Since $R^{d-1}(A)\subseteq V_E(A)$, we conclude that $R^{d-1}(A)= V_E(A)$.
\end{proof}

\begin{que}
\
\begin{enumerate}
\item
We have some evidence that the converse of Proposition \ref{mlr} is true. So we ask assuming that $R^{p}(A) = 0$ for $0\leq p \leq d-2$  and $R^{d-1}(A) = V_E(A)$ is irreducible, if then $J$ has a $d$-linear resolution.
\item
Let $J\subset E$ be an arbitrary graded ideal with $d$-linear resolution.
Is $R^{d-1}(E/J)$ always maximal (i.e. $R^{d-1}(E/J)=V_E(E/J)$) or at least irreducible?
\end{enumerate}
\end{que}

\section{Betti numbers of the linear strand of edge ideals}
\label{sect con}

In the previous section we observed in Corollary \ref{cor:2linearconj} a special case where Conjecture \ref{conB} is true.
In this section we show that this conjecture  holds for a special class of edge ideals which gives some more support for the validity of Conjecture \ref{conB}.

In the following $G$ is always a graph on a finite vertex set $V_G$ and with edge set $E_G$. For a vertex $v\in V_G$  let $\deg v$ denote the number of edges incidents to $v$. Recall that a graph $G$  is a disjoint union of complete graphs if there exist complete graphs $G_i$ such that the vertex sets
$V_{G_i}$ of $G_i$ are disjoint, $|V_{G_i}|\geq 2$, the vertex set $V_G$ of $G$ is $V_G=\bigcup_i V_{G_i}$ and the edge set $E_G$ of $G$ is $E_G=\bigcup_i E_{G_i}$. We say that $G$ has no induced 4-cycle if for every $F\subseteq V_G$ with $|F| = 4$ the induced subgraph $G_F$ of $G$ on the vertex set $F$ is not a 4-cycle.

Let $n=|V_G|$ and $E$ be the exterior algebra on $n$ exterior variables $e_1,\dots,e_n$ over a field $K$. The edge ideal $J(G)$ of $G$ is defined as $J(G)=(e_i\wedge e_j: \{\ i,j\}\in E_G)$. Before proving our main result we need the following lemma:

\begin{lem}
\label{reunicom}
Let $G$ be a disjoint union of
complete graphs and $n=|V_G|$. Then $R^1(E/J(G))$ is a union of linear subspaces and $E/J(G)$ satisfies
property (*) (see Page 2).
\end{lem}
\begin{proof}
Let $G$ be the disjoint union of complete graphs  $G_1,\dots, G_t$.
Let $r_i=|V_{G_i}|$ and so $n=\sum_{i=1}^t r_i$. Consider the edge ideals $J(G)$ and $J(G_i)$ in the exterior algebra
$E$. It is clear that $J(G)= \sum_{i=1}^{t}
J(G_i)$. The first resonance variety of $E/J(G)$ can be computed as
\begin{equation}
\label{firstreso}
R^1(E/J(G))= \{u \in E_1: u=0 \text{ or there exists } v\in E_1\text{ such that } 0\neq u\wedge v \in J(G)\}.
\end{equation}
Let $V_{G_i}=\{ i_j:
j=1,\ldots,r_i\}\subseteq [n].$
Because of $e_{i_p}\wedge e_{i_q}\in J(G_i)$ for
$1\leq p,q\leq r_i$ and Equation
(\ref{firstreso}) we have
$$
\spa_K\{e_{i_1},\ldots,
e_{i_{r_i}}\}\subseteq R^1(E/J(G)).
$$
We claim that the irreducible components of $R^1(E/J(G))$ are
exactly those vector spaces $\spa_K\{e_{i_1},\ldots, e_{i_{r_i}}\}$ for $1\leq i\leq t$. Assume that there exists an irreducible component which is not of this form. Then there exists linear forms $u, v\in E_1$ such that
$$
0\neq u\wedge v\in J(G) \text{ and } u\notin \spa_K\{e_{i_1},\ldots,
e_{i_{r_i}}\} \text{ for all } 1\leq i\leq t.
$$
Let $u=\sum_{k=1}^n \alpha_k e_{k}$ and $v=\sum_{k=1}^n \beta_k e_k$ for $\alpha_k,\beta_k \in K$.

Now we show that $\supp(u)=\supp(v)$. For this let $k_1 \in \supp(v)$ be arbitrary and choose $i$ such that $v\in V_{G_i}$. Since $\supp(u)$ is not contained in  $V_{G_i}$ there exists $k_2 \in \supp(u)$ with $k_2 \not\in V_{G_i}$.  Observe that $e_p\wedge e_q\in J(G)=\sum_{i=1}^t J(G_i)$ if and only if there is $1\leq i\leq t$ such that $p,q\in V_{G_j}$ for some $j$. So $e_{k_1}\wedge e_{k_2}\notin J(G)$. It follows that
$\alpha_{k_1}\beta_{k_2}-\alpha_{k_2}\beta_{k_1}=0$ because $u\wedge v \in J(G)$. Hence
$$
 k_1\in \supp(u),\ k_2\in \supp(v) \text{ and } \alpha_{k_2}/\beta_{k_2}=
\alpha_{k_1}/\beta_{k_1}.
$$
In particular, we see that $\supp(v)\subseteq \supp(u)$ and also
the support of $v$ is not contained on one of the $V_{G_j}$. With the same arguments it follows now that $\supp(u)\subseteq \supp(v)$
and then $\supp(u)=\supp(v)$. Moreover, we can conclude that $\alpha_k/\beta_k$ is the same constant for every $k\in \supp(u)=\supp(v)$. But then we get the contradiction
$u\wedge v=0$.

So wee see that all irreducible components of $R^1(E/J(G))$ are induced by the complete subgraphs of $G$. More precisely,
$$
R^1(E/J(G))=\bigcup_{i=1}^t \spa_K\{e_{i_1},\ldots, e_{i_{r_i}}\}.
$$
We also get that $R^1(E/J(G))$ satisfies property(*) in Page 2.
\end{proof}

\begin{lem}
\label{lemhelper}
Let $i, r$ be integers with $i,r\geq 0$. Then we have
$$
\sum_{j=0}^i \binom{i}{j}\binom{r}{j+2}=\binom{r+i}{i+2}.
$$
\end{lem}
\begin{proof}
Considering the polynomial $f(x)=(1+x)^{r+i}$ in the polynomial ring $K[x]$, we get
$$ (1+x)^{r+i}=(1+x)^i (1+x)^r=(\sum_{j=0}^i \binom{i}{j} x^j) (\sum_{t=0}^r \binom{r}{t} x^t).
$$
This implies that the coefficient of $x^{r-2}$ is $\sum_{j=0}^i \binom{i}{j} \binom{r}{r-2-j}=\sum_{j=0}^i \binom{i}{j} \binom{r}{j+2}$.
Moreover $(1+x)^{r+i}=\sum_{j=0}^{r+i} \binom{r+i}{j}x^j$, so the coefficient of $x^{r-2}$ in this equality is $\binom{r+i}{r-2}=\binom{r+i}{i+2}$. Hence we conclude that $\sum_{j=0}^i \binom{i}{j}\binom{r}{j+2}=\binom{r+i}{i+2}.$

\end{proof}
We are ready to prove the main results of this section.
\begin{thm}
\label{thmconB}
Let $G$ be a disjoint union of complete graphs  with $n=|V_G|$. Then the graded Betti numbers in the linear strand of $J(G)$ are given by
$$
\beta^E_{i,i+2}(J(G))=(i+1)\sum_{r=2}^n h_r \binom{r+ i}{i+2},
$$
where $h_r$ is the number of $r$-dimensional components of $R^1(E/J)$ in the affine space $K^n$.
\end{thm}
So Conjecture \ref{conB} is true for edge ideals of disjoint unions of complete graphs.
\begin{proof}
Let $G$ be the disjoint union of complete graphs
$G_1,\dots,G_t$ with $V_{G_i}=\{i_1,\ldots,i_{r_i}\}$, $i=1,\ldots,t$. By Lemma \ref{reunicom} we see that $R^1(E/J(G))=\bigcup_{i=1}^t
\spa_K\{e_{i_1},\ldots, e_{i_{r_i}}\}$.

Let $k_{i}(G)$ be  the number of complete subgraph on $i$ vertices of $G$. Observe that a disjoint union of complete graphs has no induced 4-cycles. It follows from \cite[Proposition 2.4]{RT}
\begin{eqnarray*}
\beta_{i,i+2}^S(I(G))&=&\sum_{v\in V_G} \binom{\deg v}{i+1}- k_{i+2}(G)=\sum_{j=1}^t\sum_{v\in V_{G_j}}\binom{\deg v}{i+1}-k_{i+2}(G)\\
        &=& \sum_{j=1}^t\sum_{v\in V_{G_j}} \binom{r_j-1}{i+1}-\sum_{j=1}^t k_{i+2}(G_j)=\sum_{j=1}^t r_j \binom{r_j-1}{i+1} -\sum_{j=1}^t \binom{r_j}{i+2}\\
        &=&\sum_{r=1}^t r \cdot h_r\binom{r-1}{i+1}-\sum_{r=1}^t h_r\binom{r}{i+2}=(i+1)\sum_{r=1}^t h_r \binom{r}{i+2}
\end{eqnarray*}
Here $S=K[x_1,\dots,x_n]$ is the polynomial ring over $K$, the ideal $I(G)=(x_i x_j: \{i,j\}\in E_G)$ is the edge ideal of $G$ over $S$
and $\beta_{i,j}^S(I(G))$ denote the graded Betti numbers of  $I(G)$ over $S$. Note that $I(G)$ is a so-called squarefree $S$-module in the sense of \cite[Definition 2.1]{YA}. Then it follows from \cite[Corollary 1.3]{RO}
that
\begin{eqnarray*}
\beta_{i,i+2}^{E}(J(G))&=& \sum_{j=0}^i
\binom{i+1}{j+1}\beta_{j,j+2}^S(I(G))=\sum_{j=0}^i \binom{i+1}{j+1} (j+1)\sum_{r=1}^t h_r \binom{r}{j+2}\\
        &=& (i+1)\sum_{j=0}^i \binom{i}{j}h_r\sum_{r=1}^t \binom{r}{j+2}=(i+1)\sum_{r=1}^t h_r\sum_{j=0}^i \binom{i}{j}\binom{r}{j+2}\\
        &=& (i+1)\sum_{r=1}^t h_r \binom{r+i}{i+2},
\end{eqnarray*}
where we get the last equality from Lemma \ref{lemhelper}. Since $\beta_{i,i+2}^{E}(J(G))=\beta_{i+1,i+2}^{E} (E/J(G))$  Conjecture \ref{conB} holds for $E/J(G)$.
\end{proof}

Conjecture \ref{conB} without property (*) is not true for an algebra $E/J$ where $J\subset E$ is an arbitrary (monomial) ideal as can be seen as follows:

\begin{exa}
\label{exconB}
Let $E=K\langle e_1,\dots,e_5\rangle$ be the exterior algebra over a field $K$. Let
$$
J=(e_1\wedge e_2, e_1\wedge e_3, e_1\wedge e_4, e_1\wedge e_5,e_2\wedge e_3\wedge e_4) \subset E.
$$
We see that $J$ is a strongly stable monomial ideal and $E/J$ does not have the property (*). By Lemma \ref{betti}, for $i\geq 0$, we have
$$
\beta^E_{i,i+2}(J)=\sum_{u\in
G(J)_2}\binom{\max(u)+i-1}{\max(u)-1}=\binom{i+1}{1}+\binom{i+2}{2}+\binom{i+3}{3}+\binom{i+4}{4}.
$$
Since $\max\{\max(u): u\in G(J)_2\}=5$, it follows from Theorem \ref{thmstrong} that
$$
R^1(E/J)=\spa_K\{e_1,\ldots,e_5\} \text{ and }
R^1(E/J) \text{ is irreducible}.
$$
By induction, we can prove
that
$$
\binom{i+1}{1}+\binom{i+2}{2}+\binom{i+3}{3}+\binom{i+4}{4}<(i+1)\binom{i+5}{i+2}
$$
for $i\geq 0$. We get
$$
\beta^E_{i+1,i+2}(E/J)=\beta^E_{i,i+2}(J)<
(i+1)\binom{i+5}{i+2}= (i+1)\sum_{r\geq 1} h_r \binom{r+i}{i+2}
$$
where $h_r$ is the number of components of $R^1(E/J)$ which have
dimension $r$ in the affine space $E_1= K^n$.
Thus Conjecture \ref{conB} does not hold for $E/J$.
\end{exa}

\end{document}